\documentclass[a4paper]{article}
\usepackage{times,cite}
\usepackage[T1]{fontenc}
\usepackage[utf8]{inputenc}
\usepackage{amsmath,amsfonts,amsthm}

\usepackage{enumitem}
\theoremstyle{plain}
\newtheorem{theorem}{Theorem}[section]

\usepackage{graphicx}
\begin{document}
\title{How not to discretize the control}

\def\N{\mathbb{N}}
\def\R{\mathbb{R}}
\def\T{\mathcal{T}}
\newcommand\norm[1]{\lVert#1\rVert}

\author{Daniel Wachsmuth\footnote{Institut f\"ur Mathematik,
Universit\"at W\"urzburg,
97074 W\"urzburg, Germany, {\tt
 daniel.wachsmuth@mathematik.uni-wuerzburg.de}} \and Gerd Wachsmuth\footnote{Technische Universität Chemnitz,
Faculty of Mathematics,
09107 Chemnitz, Germany, {\tt gerd.wachsmuth@mathematik.tu-chemnitz.de}}}
\maketitle
\begin{abstract}
In this short note, we address the discretization of optimal control
problems with higher order polynomials. We develop
a \emph{necessary and sufficient condition} to ensure that weak limits of
discrete feasible controls are feasible for the original problem.
We show by means of a simple counterexample that a naive discretization by
higher order polynomials can lead to non-feasible limits of sequences of discrete
solutions.
\end{abstract}

\section{Introduction}

We consider the discretization of the` optimal control problem
\begin{equation}\label{P}
 \min_{u\in L^2(\Omega)} J(u) \quad \text{ subject to } u \ge0.
\end{equation}
Here, $\Omega\subset \R^d$ is a bounded open set.
We set $U:=L^2(\Omega)$. The objective is given by $J: U \to \R$
and we assume that a (not necessarily unique) global solution $\bar u$ of \eqref{P} exists.
This can be guaranteed under standard assumptions on $J$.
In particular, we have in mind to choose $J$ as the reduced cost functional of a optimal control
problem subject to a partial differential equation (PDE).
For an introduction to optimal control problems for PDEs, we refer to
\cite{Troeltzsch2010:1}.
In order to numerically solve the problem, it has to be discretized.
We will investigate a particular choice of discretization, which consists of
discretizing the controls on subdivisions of the domain $\Omega$ by, e.g., piecewise
polynomial functions.
Note that we do \emph{not} address the related question ``How to not discretize the control?'',
which was popularized by Hinze \cite{hin05}.

\section{Discretization}
We consider a sequence of discretizations, indexed by an integer $n \in \N$.
We associate with each $n$ the following objects:
\begin{enumerate}[label=\bf(A\arabic*),itemindent=*]
\item\label{A1} a finite dimensional subspace $U_n\subset U$ with fixed basis $\{\phi_n^1, \ldots, \phi_n^{N_n}\}$,
 \item\label{A2} a set $\T_n$ of open, pairwise disjoint
elements $T$ with  $\bar \Omega=\overline{\bigcup_{T\in \T_n}T}$
and  $\operatorname{diam}(T)\le h_n$ for all $T\in \T_n$, where $h_n \searrow 0$,
\item\label{A3} and a functional $J_n : U_n\to \R$ approximating $J$.
\end{enumerate}
As a discretization for \eqref{P}, we choose
\begin{equation}\label{Ph}
 \min J_n(u_n)
 \quad \text{ subject to } u_n = \sum\nolimits_{i=1}^{N_n}
 \lambda_i \, \phi_n^i, \quad \text{where } \ \lambda_i\ge0 \quad \forall i = 1, \ldots, N_n.
\end{equation}
That is, the non-negativity constraint $u\ge0$ is replaced by a non-negativity constraint
on the coordinates of $u_n$ with respect to the chosen basis of $U_n$.
We assume that the discrete problem has a global solution $\bar u_n$.
In order to study convergence with respect to $n \to \infty$ we will impose the following conditions
on the sequence $\{(U_n,\T_n,J_n)\}_{n\in \N}$:
\begin{enumerate}[label=\bf(A\arabic*),itemindent=*,resume]
 \item\label{A4} $J(u)\le \liminf_{n \to \infty} J_n(u_n)$ for every sequence $\{u_n\}_{n \in \N}$ with $u_n\in U_n$ $\forall n \in \N$ and $u_n\rightharpoonup u$ in $U$ for $n \to \infty$.
 \item\label{A5}
	 $J(\bar u) \ge \limsup_{n \to \infty} J_n(v_n)$
	 for some sequence $\{v_n\}_{n \in \N}$ with $v_n \in U_n$.
\end{enumerate}
Assumptions \ref{A4}, \ref{A5} are slightly weaker than the $\Gamma$-convergence of $J_n$ towards $J$.
We remark that these assumptions are fulfilled for standard FE discretizations of
optimal control problems subject to partial differential equations \cite{hpuu,Troeltzsch2010:1}.

\section{Convergence of discrete approximations}
In this section we will show that it is sufficient that the basis functions $\phi_n^i$
have non-negative integral on the cells $T\in \T_n$ to guarantee that weak limits
of discrete solutions are feasible.
Note, that non-negativity of the basis functions is not required.

\begin{theorem}\label{theo31} Let the sequence $\{(U_n,\T_n,J_n)\}_{n \in \N}$
satisfy \ref{A1}--\ref{A5}
and, in addition,
\[
	\int_T \phi_n^i \ge0 \quad\forall n \in \N, T\in \T_n, i=1\dots N_n^i.
 \]
 Then, every weak limit $u$ of feasible points $u_n$ of \eqref{Ph} is feasible for \eqref{P},
 and the weak limit (if it exists) of global solutions $\bar u_n$ of \eqref{Ph}
 is a global solution of \eqref{P}.
\end{theorem}
\begin{proof}
	Let the feasible points $u_n$ of \eqref{Ph} converge weakly in $U$ to $u$.
	We show that $u$ is feasible for \eqref{P}.
 Let $K\subset \Omega$ be compact and non-empty. We set
 \[K_n := \bigcup_{T\in \T_n:\bar T\cap K\ne\emptyset} \bar T.\]
 Then $\int_{K_n} u_n \ge0$ for all $n \in \N$. By dominated convergence and condition \ref{A2},
 $\chi_{K_n}\to \chi_K$ in $L^2(\Omega)$, which implies
 $\int_{K_n} u_n \to \int_K u$. Hence $\int_K u \ge0$ for all such compact $K$,
 and a density argument implies $u\ge0$.

 Now, let $\bar u_n$ be globally optimal for \eqref{Ph}
 and denote by $\tilde u$ the weak limit.
 As in the first part of the proof, we can show $\tilde u \ge 0$.
 Due to \ref{A4}, \ref{A5}, we have
 $
 J(\tilde u) \le \liminf_{n \to \infty} J_n(\bar u_n) \le \liminf_{n \to \infty} J_n(v_n)
 \le\limsup_{n \to \infty} J_n(v_n) \le J(\bar u).
 $ Since $\tilde u$ is feasible for \eqref{P}, and $\bar u$ is a global solution,
 it follows that $\tilde u$ is a global solution.
\end{proof}

\section{An example with non-feasible limit}

We consider the following simple optimal control problem:
\[
 \min_{(y,u)\in H^1(\Omega)\times L^2(\Omega)} \|y-y_d\|_{L^2(\Omega)}^2 + \alpha \, \|u\|_{L^2(\Omega)}^2
\]
subject to
\[
 u\ge 0
 \]and \[\int_\Omega \nabla y\cdot\nabla v + y\,v= \int_\Omega u\,v \ \forall v\in H^1(\Omega).
\]
Here, $\Omega\subset \R^d$ is a bounded domain with polygonal boundary. We set $y_d:=-1$ and $\alpha>0$.
It is easy to check that $(\bar y,\bar u)=(0,0)$ with $J(\bar u)=|\Omega|$ is the unique solution of this problem:
due to the maximum principle it holds $y\ge0$ for all feasible pairs $(y,u)$.

We discretize this problem by finite elements on simplicial decompositions $\T_n$ of $\Omega$.
As finite element spaces we choose standard (discontinuous or continuous) Lagrange elements of polynomial degree $k\ge1$.
Further, we define
$J_n(u_n) := \norm{y_n(u_n) - y_d}_{L^2(\Omega)}^2 + \alpha \, \norm{u_h}_{L^2(\Omega)}^2$,
where $y_n(u_n)$ is the solution of a suitably discretized state equation.
Following standard arguments \cite{hpuu} it is easy to prove that \ref{A1}--\ref{A5} are
satisfied.

We denote by $(\hat\psi_1\dots \hat \psi_m)$ the Lagrange basis of order $k$ on the reference simplex $\hat T$.
We will show that if there is a basis function with negative integral, then the
solutions of the discretized problem will converge weakly to a non-feasible limit $u$.

\begin{theorem}
 Assume that $\int_{\hat T} \hat \psi_j<0$ for some $j\in \{1,\dots, m\}$.
 Denote by $\bar u_n$ the unique solution of the discretized problem.
 Then it holds
 $\bar u_n\rightharpoonup u$ (along a subsequence), where $u$ does not satisfy $u\ge0$.
\end{theorem}
\begin{proof}
We define $I_n:=\{i: \ \int_\Omega \phi_n^i<0\}$. Due to the assumption this set is non-empty.
We set $w_n :=\sum_{i\in I_n}\phi_n^i\in U_n$. Then for each element $T\in \T_n$
the function $w_n|_T$ is an affine transformation of $\hat w:=\sum_{j:\ \int_{\hat T} \hat \psi_j<0}\hat\psi_j$.
Consequently, it holds $\|w_n\|_{L^2(\Omega)}^2 = \sum_{T\in\T_n} |T| \, |\hat K|^{-1}\|\hat w\|_{L^2(\hat K)}^2 =|\Omega| \, |\hat K|^{-1}\|\hat w\|_{L^2(\hat K)}^2=:M^2$,
which shows that $\{w_n\}_{n \in \N}$ is uniformly bounded.
In addition, it holds \[\int_\Omega w_n = \sum_{T\in\T_n} |T| \, |\hat K|^{-1}\int_{\hat K} \hat w =|\Omega| \, |\hat K|^{-1}\int_{\hat K} \hat w=:-\beta<0.\]
We set $z_n := y_n(w_n)$.
Testing the discretized equation by $z_n$ and $1$ yields
$L_n := \|z_n\|_{L^2(\Omega)}\le \|w_n\|_{L^2(\Omega)}$ and
$\int_\Omega z_n =\int_\Omega w_n =- \beta<0$, respectively.

Then it holds $J_n(t\,w_n) = (t^2\,L_n^2-2\,t\,\beta+|\Omega|) + \alpha \, t^2\,M^2 \le (1+\alpha)\,M^2 \, t^2 - 2 \, t \, \beta + |\Omega|$.
For the choices $\hat t:=\frac{\beta}{(1 +\alpha) \, M^2}>0$
and $u_n:=\hat t \, w_n$ we obtain
$J_n(u_n)\le |\Omega|-\delta<|\Omega|=J(\bar u)$
with
$\delta := \beta \, \hat t > 0$.
This shows that $J_n(\bar u_n)\le |\Omega|-\delta < J(\bar u)$.

Let $\bar u_n\rightharpoonup u$ in $L^2(\Omega)$ (along a subsequence) and, consequently,
$y_n(\bar u_n) \rightharpoonup y$ in $H^1(\Omega)$.
By standard arguments, $(y,u)$ satisfy the weak formulation of the partial differential equation.
Moreover, as in the proof of Theorem \ref{theo31} it follows
$J(u)\le |\Omega| - \delta < |\Omega| = J(\bar u)$. This implies that $(y,u)$ cannot be feasible, and consequently $u\ge 0$
is violated.
\end{proof}

Numerical experiments show that basis functions with negative integral appear for sufficiently large
$k$ depending on the spatial dimension.
For the standard Lagrangian elements, we found the following situation:
\begin{enumerate}[label={$d=\arabic*$: },itemindent=*]
\item basis function have non-negative integrals for $k\in\{1,2,3,4,5,6,7,9\}$, \\but not for
 $k\in \{8,10,11\}$,
 \item basis function have non-negative integrals for $k\in\{1,2,3,5\}$, \\but not for
 $k\in \{4,6,7,8\}$,
 \item basis function have non-negative integrals for $k\in\{1,3\}$, \\but not for
 $k\in \{2,4,5,6\}$.
\end{enumerate}
Here, in particular the situation in dimension $3$ is remarkable: A naive control discretization by
$P^2$ elements with non-negativity constraints on the coefficients may fail to produce
approximations with feasible limits!

Let us remark that similar results can be proven for problems with homogeneous
Dirichlet boundary conditions
and for more general discretizations using an affine family of finite elements.

\newenvironment{acknowledgement}{\begin{small}\paragraph{Acknowledgements}}%
                           {\end{small}}
\begin{acknowledgement}
The first author was partly supported by DFG grant Wa 3626/1-1.
\end{acknowledgement}

\vspace{\baselineskip}

\bibliography{pamm}
\bibliographystyle{plain}

\end{document}